\documentclass[11pt,a4paper, envcountsame]{amsart}

\usepackage[usenames,dvipsnames]{color}

 \usepackage[colorlinks,citecolor=blue,urlcolor=blue, linkcolor=blue]{hyperref}

 \usepackage{amsmath, amssymb, xspace}
 \usepackage{graphicx}

 \usepackage[all]{xy}

\usepackage{tikz}
\usetikzlibrary{arrows,snakes,positioning,backgrounds,shadows}

\newtheorem{theorem}{Theorem}[section]

\newtheorem{thm}[theorem]{Theorem}

\newtheorem{lemma}[theorem]{Lemma}

\newtheorem{question}[theorem]{Question}

\theoremstyle{definition}
\newtheorem{definition}[theorem]{Definition}

\newcommand{\NN}{{\mathbb{N}}}
\newcommand{\RR}{{\mathbb{R}}}

\newcommand{\ZZ}{{\mathbb{Z}}}
\newcommand{\sub}{\subseteq}
\newcommand{\sN}[1]{_{#1\in \NN}}
\newcommand{\uhr}[1]{\! \upharpoonright_{#1}}

\newcommand{\PI}[1]{\Pi^0_{#1}}

\newcommand{\bi}{\begin{itemize}}
\newcommand{\ei}{\end{itemize}}
\newcommand{\bc}{\begin{center}}
\newcommand{\ec}{\end{center}}

\newcommand{\ES}{\emptyset}
\newcommand{\estring}{\la \ra}

\newcommand{\tp}[1]{2^{#1}}
\newcommand{\ex}{\exists}
\newcommand{\fa}{\forall}

\newcommand{\la}{\langle}
\newcommand{\ra}{\rangle}

\newcommand{\seqcantor}{2^{ \NN}}

\newcommand{\cantor}{\seqcantor}

\newcommand{\n}{\noindent}

\newcommand{\sss}{\sigma}

\newcommand{\lland}{\, \land \, }

\newcommand \seq[1]{{\left\langle{#1}\right\rangle}}

\newcommand\+[1]{\mathcal{#1}}

\newcommand{\ol}{\overline}

\newcommand{\lra}{\leftrightarrow}
\newcommand{\LR}{\Leftrightarrow}
\newcommand{\RA}{\Rightarrow}

\newcommand{\sssl}{\ensuremath{|\sigma|}}

\renewcommand{\hat}{\widehat}
 \newcommand{\pp}[3]{\prod_{#2 \in #3}   {x_#2}^{{#1}_#2}} 
 \newcommand{\si}[1]{ \{ #1 \}}
 
\begin{document}

\title[Complexity of isomorphism between   profinite groups]{The complexity of   isomorphism between  countably based profinite groups}

 \author{Andr\'e Nies}
\begin{abstract} A topological group $G$ is  profinite if it is compact and totally disconnected. Equivalently, $G$ is  the  inverse limit of a surjective  system of finite groups carrying  the discrete topology. We discuss how to represent a countably based profinite group as a point in a Polish space. Then we  study the complexity of isomorphism using the theory of Borel reducibility in descriptive set theory. For topologically finitely generated profinite groups this complexity is the same as the one of identity for reals. In general, it is the same as the complexity of  isomorphism for countable graphs.  \end{abstract}

\maketitle

\section{Introduction}

A compact topological group $G$ is called profinite if  the clopen sets form a basis for the topology.  Equivalently, the open normal subgroups form a  base of nbhds of  the identity. Since open subgroups of a compact group have finite index, this means that  $G$ is the inverse limit of its  system of finite quotients with the natural projection maps. 

All profinite groups in this paper will be countably based without further mention.  This makes it possible to represent them  as  points in a Polish space.  
If $X,Y$ are Polish spaces and $E, F$ equivalence relations on $X,Y$ respectively, one writes $(X,E) \le_B (Y,F)$ (or simply $E \le_B F$) if there is a Borel function $g \colon X \to Y$ such that $uEv \leftrightarrow gu F gv$ for each $u,v \in X$.  

The goal of this paper is to study the complexity of (topological) isomorphism for various classes of profinite groups.

For topologically finitely generated profinite groups,  the complexity of isomorphism  turns out to be rather low down: it  is the same as the one of identity for reals. For the upper bound we   rely on  a result of Lubotzky 
\cite{Lubotzky:01} to show that isomorphism is the orbit equivalence relation of the action of a compact group, hence closed, and hence Borel below identity on $\RR$.  To show hardness we use another result of Lubotzky 
\cite{Lubotzky:05}.

  The lowness of isomorphism for  finitely generated  profinite groups  contrasts with the complexity of isomorphism for abstract finitely generated groups, which is as hard as possible: by a result of Thomas and Velickovic~\cite{Thomas.Velickovic:99} it is Borel complete for  Borel equivalence relations with all classes countable.

$S_\infty$ denotes the Polish group of permutations of $\omega$. 
An $S_\infty$ orbit equivalence relation is an orbit equivalence relation of a continuous action $S_\infty$ on a Polish space.    We show that isomorphism of general profinite groups is  complete for $S_\infty$ orbit equivalence relations. So, in a set theoretic   sense,  the profinite groups, being either finite or of size the continuum,  behave like countable structures; for instance,  isomorphism of countable graphs is well-known to be $S_\infty$ complete.

\section{Preliminaries}
In a group that is finitely generated as a profinite  group, all subgroups of finite index are open \cite{Nikolov.Segal:07}. This deep theorem implies that the topological structures is determined by the group theoretic structure. In particular, all abstract homomorphisms between such groups are continuous.

\subsection{Completion} \label{ss:completion}
 We follow \cite[Section 3.2]{Ribes.Zalesski:00}. 
 Let $G$ be a group, $\+ V$ a set of normal subgroups of finite index in $G$ such that $U, V \in  \+ V $ implies that there is $W \in \+ V$  with  $W \sub U  \cap V$. We can turn $G$ into a topological group by declaring $\+ V$ a basis of neighbourhoods (nbhds) of the identity. In other words, $M \sub G$ is open  if for each $x \in M$ there is $U \in \+ V$ such that $xU \sub M$. 
 
 The completion of $G$ with respect to $\+ V$ is the inverse limit  $$G_{\+ V} = \varprojlim_{U \in \+ V} G/U,$$ where $\+ V$ is ordered under inclusion and the inverse system is equipped with  the natural    maps: for $U \sub  V$, the  map  $p_{U,V} \colon G/U \to G/V $  is given by $gU \mapsto gV$. The inverse limit can be seen as a closed subgroup of the direct product $\prod_{U \in \+ V} G/U$ (where each group $G/U$ carries  the discrete topology), consisting of the functions~$\alpha $ such that $p_{U,V}(\alpha (gU)) = gV$ for each $g$.  Note that the map $g\mapsto (gU)_{U \in \+ V}$ is a continuous homomorphism $G \to G_{\+ V}$ with dense image; it is injective iff $\bigcap \+ V = \{1\}$. 
 
 If the set $\+ V$ is understood from the context, we will usually  write $\widehat G$ instead of $G_\+ V$.
 
\subsection{The Polish space of  profinite groups} \label{ss:Pol hyperspace}

\begin{definition} \label{def: F_finite }  Let $   \widehat F_k$ be the free profinite group in $k$ generators $x_0,   \ldots, x_{k-1}$  ($ k <  \omega$).  \end{definition} Thus, $   \widehat F_k$ is the profinite completion of   the abstract free group on $k$ generators.  Any topologically finitely generated  profinite group can be written in  the form \[    \widehat F_k / R\]
for some $k$ and   a closed normal subgroup~$R$ of~$   \widehat F_k$.

\begin{definition} \label{def: F_omega} Let $   \widehat F_\omega$ be the free profinite group on a sequence  of  generators $x_0, x_1,  x_2 \ldots $   converging to $1$ \cite[Thm.\ 3.3.16]{Ribes.Zalesski:00}. \end{definition} 
 Thus, $   \widehat F_\omega$ is the   completion   in the sense of Subsection~\ref{ss:completion} of  the   free group $F_\omega$ on   generators $x_0, x_1, \ldots$ with respect to  the system of normal   subgroups of finite index  that contain almost all the $x_i$.  Any   profinite group $G$   has a   generating sequence $\seq{g_i}\sN i$ converging to $1$.  This is easy  to see using coset representatives for a descending  sequence of open normal subgroups that form a fundamental system of nbhds of $1_G$, as in the proof of Thm.~\ref{thm:class ctble} below. (Also  see \cite[Prop. 2.4.4 and 2.6.1]{Ribes.Zalesski:00}.) By the universal property of the completion, the map from the abstract free group induced by $x_i \to g_i$ extends to a continuous epimorphism $ \widehat F_\omega \to G$. So $G$   can be written in  the form \[    \widehat F_k / R\]
where $R$ is a closed normal subgroup of $   \widehat F_k$.  

In the following let $k \le \omega$. For a compact Polish group $G$ (such as $   \widehat F_k$) let  $ \+ N(G)$ denote  the collection  of normal closed subgroups of $G$. It is standard to equip $ \+ N(G)$ with the structure of a Polish space, in fact of a compact countably based space, as we describe next.

The compact  subsets of a complete separable metric  space  $(M,\delta)$ form a complete metric  space $\+ K(M)$ with the usual Hausdorff distance \bc $\delta_H(A,B) = \max (\sup_{x\in A} \inf_{y\in B} \delta(x,y), \sup_{y\in B} \inf_{x\in A} \delta(x,y))$. \ec  Let $D \sub M$ be countable dense. Then $\+ K(M)$  contains as a dense subset  the set  of finite subsets of $D$.  Since $\+ K(M)$ is a metric space, this implies that $\+ K(M)$  is countably based.
   If $M$ is compact then $\+ K(M)$ is   compact as well.

 It is well-known  that every compact Polish group $G$ has a compatible bi-invariant metric $\delta$; that is, $\delta$ induces the given topology, and $\delta(xg, yx) = \delta(gx,gy) = \delta(x,y)$ for any $x,y, g \in G$.

We verify that $ \+ N(G)$ is  closed in  $\+ K(G)$, using that   the metric on $G$  is bi-invariant. Firstly, if a sequence of    subgroups $\seq {U_n}\sN n$ in  $\+ K(G)$ converges  to $U\in \+ K(G)$, then $U$ is a subgroup. For suppose  $a,b \in U$. For each $\epsilon > 0 $, for sufficiently  large $n$  we can choose $a_n, b_n \in U_n$  with $d(a_n, a) < \epsilon$ and $d(b_n, b) < \epsilon$. Then $a_nb_n \in U_n$ and $d(a_nb_n, ab) \le d(a_nb_n, a_n b) + d(a_nb, ab) < 2\epsilon$. Secondly,  for each $c\in G$ the conjugation  map $A  \to A^c$ is an isometry of $\+ K(G)$. If  all the $U_n$ are normal in $G$, then   $U^c = (\lim_n U_n)^c = \lim_n U_n = U$.  Hence $U$ is normal. 

(One can avoid the  hypothesis that  $G$ is compact, as long as there is a bi-invariant  metric compatible with the topology: the closed normal subgroups of a Polish group $G$ form a Polish space, being a closed subset of the Effros space $\+ F (G)$ of non-empty closed sets in $G$. While this space is usually seen as  a Borel structure, it can be topologized using  the Wijsman topology,  the weakest topology that makes all the maps $C \to d(g,C)$, $C\in \+ F (G)$, continuous.)

We conclude that $\+ N (   \widehat F_k)$ can be seen as a compact Polish space, with a compatible (complete) metric derived from a bi-invariant metric on $   \widehat F_k$.  We will refer to $\+ N (   \widehat F_k)$ as the \emph{space of $k$-generated profinite groups}. 

 For a profinite  group $G$, by $ \mathtt{Aut}(G)$ we denote the group of continuous automorphisms. (Note that by compactness,  the inverse iof a continuous automorphism is continuous as well. Also,  if $G$ is     finitely generated as a profinite group, then every automorphism is continuous by the result of  \cite{Nikolov.Segal:07} mentioned above.) As described in  \cite[Ex.\ 6 on page 52]{Wilson:98},   for any finitely generated profinite group $G$, the group $Q= \mathtt{Aut}(G)$ is compact, and in fact profinite.    To see this,  let $V_n$ be the intersection of the finitely many open subgroups of $G$ of index at most $n$. Note that  $V_n$ is open  and invariant in $G$, and the $V_n$ form a basis of nbhds for $1_G$. Since the sequence  $\seq {V_n}\sN n$  is a base of nbhds of $1$, we can use it  to define a compatible bi-invariant ultrametric on $G$:  $\delta(g,h) = \inf \{  \tp{-n} \colon gh^{-1} \in V_n\}$.   
 Note that $\delta_H(R,S) =  \inf \{  \tp{-n} \colon \, RV_n= SV_n\}$.

 Now let $W_n \unlhd Q$ be the normal subgroup consisting of the $\theta \in Q$ that induce the identity on $G/V_n$. Clearly $ \bigcap_n W_n = \{1\}$. Taking  $\seq {W_n}\sN n$ as a fundamental system of nbhds of $1$ in $Q$ yields the desired topology on~$Q$.


 \begin{lemma} \label{lem:action Q G} Let $G$ be a  finitely generated profinite group $G$, and let   $Q= \mathtt{Aut}(G)$.

\n 
(i) The   natural action of $Q$ on $G$ is continuous. 

\n (ii) The   natural action of $Q$ on $\+ N(G)$ given by  $ \theta \cdot R =  \theta(R)$  is continuous.  \end{lemma}

\begin{proof} (i)  Since $Q$ is a Polish group and $G$ a Polish space, it suffices to show that the action is separately continuous, namely, when we fix one argument the resulting unary map is continuous  \cite[I.9.14]{Kechris:95}. Firstly, each $\theta \in Q$ is (uniformly) continuous as a map $G \to G$. Secondly, for each $g \in G$,  if $\theta \chi^{-1} \in W_n$, then $\delta(\theta g, \chi g) < \tp{-n+1}$, so for any $g\in G$ the map $Q \to G$ given by $\theta \to \theta g$ is continuous.

(ii)  We verify separate continuity. Firstly, any  $\theta \in Q$ can be extended to a continuous map $\+ K (G) \to \+ K(G)$, and hence its restriction is a continuous map $\+ N(G) \to \+ N(G)$.  Secondly, for each $R \in \+ N(G)$,  if $\theta \chi^{-1} \in W_n$ then by definition of the Hausdorff distance we have $\delta_H(\theta(R), \chi(R) < \tp{-n+1}$. 
\end{proof}

Observe  that $\mathtt{Aut}(\hat F_\omega)$ is not compact as  $S_\infty$  embeds into it  as a  closed subgroup. However, $\mathtt{Aut}(\hat F_\omega)$ is non-archimedean: the open subgroups  $W_n$ defined above form  a nbhd base of $\{1\}$. So it is     isomorphic to  a closed subgroup of $S_\infty$.  

\section{Complexity of isomorphism   between   finitely generated profinite groups}

We view the disjoint union $\bigsqcup_{k< \omega} \+ N (  { \widehat F_k})$ as the space of finitely generated profinite groups. Note that this space becomes a    Polish space    by declaring  a set $U$ open if $U \cap \+ N (  { \widehat F_k})$ is open for each $k < \omega$. 
\begin{thm} \label{thm:smooth} The isomorphism relation $E_{f.g.}$ between  finitely generated profinite groups is Borel equivalent to $\mathtt{id}_\RR$, the identity equivalence relation  on $\RR$.  \end{thm}
\begin{proof}   
We begin by   showing  that $E_{f.g.} \le_B \mathtt{id}_\RR$; this property of an equivalence relation on a Polish space is called smoothness (e.g.\ \cite[Section 5.4]{Gao:09})  

An easy proof  uses that a f.g.\ profinite group is given by the set of its finite quotients  \cite[16.10.7]{Fried.Jarden:06}.
Instead we provide a self-contained  proof that could  still help for the computable case.
 We thank Alex Lubotzky for pointing out the crucial fact in \cite[Prop.\ 2.2]{Lubotzky:01} used below to show this   smoothness of  $E_{f.g.}$.

Firstly   let us   consider  the case of a  fixed finite number  $k $ of generators.  Write $G=    \widehat F_k$ and as before let   $Q= \mathtt{Aut}(G)$.

 For $S, T \in \+ N(   \widehat F_k)$, we have \[    \widehat F_k/S \cong    \widehat F_k/ T \lra  \ex \theta \in Q \, [ \theta(S) = T] \]
by \cite[Prop.\ 2.2]{Lubotzky:01} 
(This depends on a  Lemma of Gasch\"utz   on lifting generating sets of finite groups. See e.g.\ \cite[17.7.2]{Fried.Jarden:06}.) 
     The natural action of $Q$ on $\+ N(   \widehat F_k)$ is continuous by Lemma~\ref{lem:action Q G}. Then, since $Q$ is compact,  the  orbit equivalence relation of this  action on  $ \+ N(   \widehat F_k)$ is closed, and hence smooth; see e.g.\  \cite[5.4.7]{Gao:09}.

In the general case,   two  given profinite groups may have a  different number of generators, say $r <  k< \omega$. 

\begin{lemma}  There is a continuous   embedding $\phi \colon \, \+ N(   \widehat F_r) \to \+ N(   \widehat F_k)$ such  that $ \widehat F_r/U \cong  \widehat F_k/\phi (U)$ for each $U \in \+ N( \widehat F_r)$.   \end{lemma}
Assuming this,     by the case of $k$ generators and taking a preimage under $\phi$,  the relation $\{ \la U, S \ra \colon    \widehat F_r / U \cong    \widehat F_k / S\}$ is closed as a subset of  $\+ N(   \widehat F_r) \times \+ N(   \widehat F_k)$. This shows that 
$E_{f.g.} $ is closed on $\bigsqcup_{k< \omega} \+ N (  { \widehat F_k})$.

We verify the lemma. The embedding is as expected from the case of abstract finitely generated groups. Let $P$ be the closed normal subgroup of $\hat F_k$ generated by $\{x_{r}, \ldots, x_{k-1} \}$. Then $\hat F_k/P = \hat F_r$ and hence $\hat F_k = \hat F_r P$. 

For $U \in \+ N( \widehat F_r)$ let  $\la U \ra$ denote the closed normal subgroup of $\hat F_k$ generated by $U$. Note that $\la U \ra $ is the closure of the set of finite products $\prod_{i=1}^n u_i^{p_i}$ for $u_i \in U$ and $p_i \in P$. 

Define  $\phi(U) = \la U \ra P$.  Note that $\phi(U) \in \+ N(\hat F_k)$ because the product of two closed normal subgroups is closed again. 
As in Subsection~\ref{ss:Pol hyperspace} above, let $V_n$ be the intersection of the finitely many open subgroups of $\hat F_r$ of index at most $n$, and let $\delta$ be the corresponding distance.  To show $\phi$ is continuous, given $L \lhd_o \hat F_k$ let $n$ be so large that $V_n \le L \cap \hat F_r$. Suppose   $U, V \in \+ N(\hat F_r)$ and   $\delta_H(U,V)\le \tp{-n}$ where $\delta_H$ denotes the Hausdorff distance in $\hat F_r$. Then $U (L \cap \hat F_r) = V (L \cap \hat F_r)$. We show that $\phi(U) L = \phi(V) L$. By symmetry, it suffices to verify that $u^p \in \phi(V) L$ for each $u \in U$ and $p \in P$. Now $u = v\ell$ for some $v \in V$ and $\ell \in L$. Then $u^p= v^p \ell^p \in \la V \ra P L = \phi(V) L$. This shows that $\phi$ is uniformly continuous and completes the lemma.

 
We prove  next the converse relation $E_{f.g.} \ge_B \mathtt{id}_\RR$. A profinite group $H$  is called \emph{finitely presented} if $H =    \widehat F_k / R$, $k < \omega$, and $R$ is finitely generated as a closed normal subgroup of $   \widehat F_k$. To show that $\mathtt{id}_\RR \le_B E_{f.g.}$, we use an  argument  of Lubotzky~\cite[Prop 6.1]{Lubotzky:05} who showed  that there are continuum many non-isomorphic  profinite groups that are finitely presented  as profinite groups. For a set $P$ of primes  let 
 $$G_P= \prod_{p\in P} \mathtt{SL}_2(\ZZ_p)= \mathtt{SL}_2( \hat \ZZ)/ \prod_{q\not \in P} \mathtt{SL}_2(\ZZ_q).$$
Here $\ZZ_p$ is the profinite ring  of $p$-adic integers, and $\hat   \ZZ$ is the completion of $\ZZ$, which is isomorphic to $\prod_{p \, \text{prime}} \ZZ_p$. The second equality shows that $G_P$ is finitely presented as a profinite group. Clearly the map  $P \to G_P$ is Borel, and $P= Q \lra G_P \cong G_Q$. 
\end{proof}
We note that Silver's dichotomy theorem, e.g.\ \cite[5.3.5]{Gao:09},  implies that any equivalence relation that is  strictly Borel below $\mathtt{id}_\RR$ has  countably many classes. So  the plain result of Lubotzky~\cite[Prop 6.1]{Lubotzky:05} now already yields  the Borel equivalence $E_{f.p.} \equiv_B E_{f.g.} \equiv_B \mathtt{id}_\RR$. However, by the proof of the result explained  above,  we in fact don't need Silver's result.

 Matthias Aschenbrenner has  pointed out a connection to a problem due to Grothendieck. A \emph{Grothendieck pair} is an embedding $u \colon G \to H$ of non-isomorphic f.p.\ residually finite  groups  such that the profinite completion $\hat u \colon \hat G \to \hat H$ is an isomorphism. Bridson and Grunewald \cite{Bridson.Grunewald:04} showed  that such pairs exist, thereby answering Grothendieck's question.  

 A Borel equivalence relation $E$ with all  classes countable   is called  \emph{weakly universal} if for  each such $F$, there is a Borel function $g$ such that $xFy \to g(x) E g(y)$, and for each $z$ the preimage $g^{-1}([z]_E)$ contains at most countably many $F$-classes.  One can slightly  modify  the usual proof, based on the 0-1-law in measure theory, that almost equality of infinite bit sequences  is not smooth, in order to  show that no weakly universal equivalence relation is smooth. 

Jay Williams \cite{Williams:15} has proved that the isomorphism relation for  f.g.\ groups of solvability class~$3$ is weakly universal. In general solvable f.g.\ groups  are not always residually finite; in fact they can have unsolvable word problem. However it is possible that   Jay's groups are r.f (which needs to be checked). If so,  then since   the process of profinite completion is Borel,  there are  now   non-isomorphic residually finite f.g.\ groups $G,H$ with isomorphic profinite completions via  Theorem~\ref{thm:smooth}.

\section{Towards complexity of isomorphism for  profinite groups}  
We now consider profinite groups that    aren't necessarily finitely generated.
As before, we think of such a group as being given by a   presentation $    \widehat F_k / N$, $N$ closed, where now  $k= \omega$. Note that   one has to explicitly require that   isomorphisms are continuous (while  continuity holds  automatically  for  algebraic homomorphisms between f.g.\ profinite groups). 

For a profinite group $G$, the commutator subgroup $G'$ is the least closed normal subgroup $S$ such that $G/S$ is abelian. The  closed normal subgroups $N \in \+ N(   \widehat F_\omega)$ such that  $(   \widehat F_\omega) ' \sub N $   form a closed subset of $\+ N(   \widehat F_\omega)$. To see this, note that  as the usual algebraic commutator subgroup   $ F_\omega' $ is dense in $  (\widehat F_\omega) ' $,  it suffices to require that  $ F_\omega' $ is contained in $N$. In this way  one obtains the space $\+ N_{ab} (   \widehat F_\omega)$ of  presentations of   abelian  profinite groups.  

 As pointed out by A.\ Melnikov, even isomorphism of these  groups  is quite complex.   Pontryagin  duality  (see e.g.\ \cite{Hofmann.Morris:06}) is a contravariant  functor on the category of abelian locally compact   groups $G$,  that associates  to   $G$     the  group $  G^*$ of continuous homomorphisms from $G$  into the unit circle $\mathbb T$,  with the compact-open topology (which coincides with  the topology inherited from the product topology if $G$ is discrete). For a morphism $\alpha \colon G \to H$ let $\alpha^* \colon H^* \to G^*$ be the morphism defined by $\alpha^*(\psi) = \alpha \circ \psi$. 
 
 The Pontryagin  duality theorem says that     for each~$G$  we have  $ G\cong ({G^*})^* $  via the   map that sends $g\in G$ to the map $\theta \to \theta(g)$.  A special case of this states that (discrete) abelian   torsion  groups $A$ correspond to  abelian   profinite groups (see \cite[Thm.\ 2.9.6]{Ribes.Zalesski:00} for a self-contained proof of this   case).  Then, as $A$ ranges  over    the abelian countable torsion  groups, $A^*$ ranges over  the  abelian  profinite groups. We have  $A \cong B$ iff $A^* \cong B^*$. The duality functor and its inverse are Borel with these restrictions on  the domain and range. Therefore  the  isomorphism  relation  between abelian countable torsion groups is Borel equivalent to continuous isomorphism between    abelian profinite groups. 
 
 By Friedman and Stanley \cite{Friedman.Stanley:89}, the  isomorphism  relation  between abelian countable torsion groups is  strictly in between $E_0$ (a.e.\ equality of infinite bit sequences, which is Borel equivalent to isomorphism of rank 1 abelian groups)  and graph isomorphism (an $S_\infty$ complete orbit equivalence relation). It  is closely related to the equivalence relation $\mathtt{id}(2^{< \omega_1})$ discussed in \cite[Section 9.2]{Gao:09}, which is modelled on the  classification of  countable abelian torsion groups via Ulm invariants.  Also see the diagram \cite[p.\ 351]{Gao:09} which shows that $\mathtt{id}(2^{< \omega_1})$ is strictly   between $E_0$ and graph isomorphism.

\section{Describing profinite groups by filters on a countable lattice} 


In this section we introduce  another way of representing  profinite groups as points in a Polish space. We view them as filters on the lattice  $\mathbb P = (P, \le,  \cap , \cdot) $  of open normal subgroups of $\hat F_\omega$, where
\[  P = \{ L \unlhd_o \hat F_\omega \colon  x_i \in L \, \text{for almost every} \, i \}. \]
This is closer to the view of profinite groups as inverse limits of finite groups. Besides an application in Section~\ref{s:class countable structures}, this view is also more appropriate for an analysis of the complexity of the isomorphism relation between  profinite groups using the tools of computability theory.

When defining the lattice  $\mathbb P$ we can equivalently replace $\hat F_\omega$ by the free group $F_\omega$ with the topology given after Def.\ \ref{def: F_omega}.
For, by \cite[Prop 3.2.2]{Ribes.Zalesski:00} we have $L = \overline{ (F_\omega \cap L)}$  for each   $L\unlhd_o \hat F_\omega$. Also, for each $N \unlhd_o F_\omega$ containing almost all the generators, we have $N = \ol N \cap F_\omega$.

  We can effectively list without repetitions all the epimorphisms $\phi \colon F_\omega \to B$ where  $B$ is a finite group such that $\phi(x_i)= 1$ for almost every $i$. Each $L\in P$, viewed as a subgroup of $F_\omega$,  has the form $\ker \phi$ for such a $\phi$, and hence   can be described  effectively by a single natural number in such a way that the ordering  and lattice  operations are computable. We denote by $L_i$ the element of $\mathbb P$ described by $i$.   So we can also view $\mathbb P$ as a computable lattice defined on $\omega$. Without loss of generality we may assume that $L_0 > L_2 > L_4> \ldots$ is a base of nbhds of $1 \in F_\omega$ for the topology on $F_\omega$ we given after Definition~\ref{def: F_omega}. As usual we obtain  a compatible ultrametric  $\delta$ on $F_\omega$ from  this sequence:  $$\delta(g,h)= \inf \{ \tp{-n} \colon \, gh^{-1} \in L_{2n}\}.$$  

Let $\+ Q(\mathbb P)$ denote the set of filters of $\mathbb P$. By the remark above we can view $\+ Q(\mathbb P)$ as an effectively closed (that is, $\PI 1$) set in Cantor space $\cantor$, and hence as a Polish space.  In the following the variables  $R,S$ range over elements of $\+ N (\hat F_\omega)$, and $L$ over elements of $P$. For $Z \in \cantor$ by $Z\uhr n$ we denote the string consisting of the first $n$ bits of~$Z$. If $\sss$ is  a string of length $n$, by $[\sss]$ we denote the set   $\{ Z \in \cantor  \colon \, Z \uhr n =\sss\}$.
\begin{lemma} \mbox{} 

 $\delta_H(R,S) \le \tp{-n} \LR  R L_{2n} = S L_{2n} \LR \fa L \ge L_{2n} [R \le L \lra S \le L]$. \end{lemma}
\begin{proof} For the first equivalence,  
clearly,  for each $u \in R$ there is $v \in S$ such that $uv^{-1} \in L_{2n}$ iff  $R \le S L_{2n}$. Symmetrically, 
for each $v \in S$ there is $u \in R$ such that $uv^{-1} \in L_{2n}$ iff    $R \le S L_{2n}$.

 The second equivalence   is clear because $R L_{2n}$ is the least $L \ge R, L_{2n}$. 
\end{proof}

\begin{lemma} \label{homeo} There is a continuous isomorphism   of  lattices  $\Phi \colon \+ N (\hat F_\omega) \to \+ Q(\mathbb P)$ given  by
\begin{eqnarray*}  \Phi(R) & =  & \{ L\in P \colon N \le L\}. \text{ Its inverse is} \\ 
\Theta(\+ G)  & = &  \bigcap \+ G. \end{eqnarray*} 
\end{lemma}
\begin{proof} Clearly $\Phi$ and $\Theta$ are inverses. To show that  $\Phi$ is (uniformly) continuous, given $ s \in \omega$, let $n$ be least so that $L_{2n} \le L_i$ for each $i< s$. We have $\delta_H(R,S) \le \tp{-n} \LR R L_{2n} = S L_{2n} \RA \Phi(R) \in [\sss] \lra \Phi(S) \in [\sss]$ for each string $\sss$ of length $s$. 

This already implies  that $\Theta$ is (uniformly) continuous by compactness of the two spaces. However, we can also obtain explicit bounds for the uniform continuity: given $n$, suppose  $Z, W \in \+ Q(\mathbb P)$ and $Z\uhr {r+1} = W\uhr {r+1}$ where  $r= \max \{i \colon  \,  L_i \ge L_{2n}\}$.  Then $\Theta(Z) L_{2n}  = \Theta(W) L_{2n}$ and hence   $\delta(\Theta(Z), \Theta(W))  \le \tp{-n}$.
\end{proof}

\section{An upper bound on the complexity of isomorphism between  profinite groups}
\label{s:class countable structures}

 \begin{thm}  \label{thm:class ctble} Isomorphism of profinite groups is classifiable by countable structures. 
\end{thm}
\begin{proof} 
Let $\cong$ denote isomorphism of profinite groups, viewed as a relation on $\+ N (\hat F_\omega)$. We will determine  a closed subgroup $V $ of $S_\infty$ and a continuous  $V$-action on a Polish space $X$  such that  $(\+ N (\hat F_\omega), \cong) $ is Borel below the orbit equivalence relation $E^X_V$. This will suffice for the theorem by a  folklore   fact from  descriptive set theory provided in Lemma~\ref{lem:subgroup}  below. 

To obtain $V$ and its action  on a Polish space $X$, we use   that  every  infinite profinite group $G= \hat F_\omega/ R$   is homeomorphic to  Cantor space $\cantor$. So we can view the group operations as continuous functions on Cantor space. $X$~is the space of continuous binary functions on Cantor space that encode  a ``group difference operation" $\rho(a,b) =  ab^{-1}$ where $0^\omega$  is the identity.  $V$ is the group of homeomorphisms of Cantor space that fix $0^\omega$.   $V$ acts continuously on $X$. The group $V$   can be seen naturally as a closed subgroup of $S_\infty$ via the usual Stone duality.  Given an infinite profinite group $G = \hat F_\omega/ R$, the task is to obtain a homeomorphism of $G$ with Cantor space sending the identity to $0^\omega$. This   yields  a homeomorphic copy $\rho_R$ of the group difference operation as a continuous binary function  on Cantor space.  Once we can do this, we have $\hat F_\omega/ R \cong \hat F_\omega/ S$ $\LR$  there is $g\in V$ sending $\rho_R $ to $\rho_S$, as required. 

We describe how to accomplish this task. Via Lemma~\ref{homeo} we may view $R \in \+ N (\hat F_\omega)$ as a filter on $\mathbb P$. The principal filters form a countable  ($F_\sss$) set in $\+ Q(\mathbb P)$, which we can ignore in the Borel reduction. So we  may further assume that $R$ is non-principal. 
Using  $R$ as a Turing oracle  we may effectively obtain a sequence $F_\omega = L_{i_0} > L_{i_1} > \ldots $ generating $R$ as a filter. In other words,   $\bigcap_n L_{i_n} = R$ when $R$ is viewed as an element of $\+ N (\hat F_\omega)$. 

We write $S_n = L_{i_n}$. To obtain the homeomorphism of $\hat F_\omega/ R$ as a topological space  with Cantor space, for each 
$n$ we can effectively determine  $k_n = |S_n : S_{n+1}|$ and  a sequence  $\seq{ g^{(n)}_{i}}_{i< k_n}$ of coset representatives for $S_{n+1}$ in $S_n$ such that  $g^{(n)}_0 = 1$.  

Let $T$ be the tree of   strings  $\sss \in \omega^{< \omega}$  such that  $\sss(i) < k_i $ for each $i< \sssl$. For $\sssl = n$  we have  a  coset \begin{equation} \label{eqn:coset} C_\sigma = g^{(0)}_{\sss(0)} g^{(1)}_{\sss(1)} \ldots g^{(n-1)}_{\sss(n-1)} S_n.\end{equation} 
The  clopen sets $C_\sss$ form  a basis for    $\hat F_\omega/ R$. In this way  $\hat F_\omega/ R$ is naturally homeomorphic to  $[T]$  where   the identity element corresponds to  $0^\omega$.

For $Z \in [T]$, by $Z\uhr n$ we denote the string consisting of the first $n$ entries  of~$Z$. The group  difference operation   $a,b \to ab^{-1}$ of $G= \hat F_\omega/ R$ can now be seen as a continuous map $[T]\times [T] \to [T]$ given by  $(Z,W) \mapsto Y$ where,  via  the   identifications    of strings of length $n$ with their  coset representatives for    $F_\omega/S_n$ given by (\ref{eqn:coset}),  we have  \bc $(Z \uhr n  S_n) (W \uhr n S_n) =  Y\uhr n S_n $  for each $n$.  \ec 

The standard homeomorphism $[T] \to \cantor $ is defined by $Z \to \bigcup_n F(Z \uhr n)$, where $F$ is the computable map defined as follows. 
Let $F(\estring) =  \estring$.
 If  $F(\sss)$ has been  defined where  $\sssl =n$,  let $F(\sss 0) =  F(\sss) 0^{k_n}$ and $F(\sss i) = F(\sss) 0^{k_n -i} 1$ for  $0< i < k_n$. Note that $0^\omega$ is mapped to $0^\omega$. Via this homeomorphism the group  difference operation  is turned  into a continuous map $\rho_R \colon \cantor \times \cantor \to \cantor$. 

Cantor space $\cantor$ is equipped with the usual ultrametric~$d$. The space of continuous functions $\+ C (\cantor \times \cantor, \cantor)$  is a  Polish space via the supremum  distance based on $d$. Let $X$ be the  closed  subspace of this space consisting of the binary operations $\rho$  that satisfy the group axioms, written in terms of the group difference operation, with $0^\omega$ as the identity element. (That is, substitute $\rho(0^\omega,a)$ for the inverse operation, and $\rho (a, \rho(0^\omega, b))$ for the group operation.) The desired Polish group $V$ is the stabiliser of $0^\omega$ in the Polish group $\+ H(\cantor)$, the homeomorphisms of Cantor space with the topology inherited from $\+ C(\cantor, \cantor)$, which canonically acts on $X$ via \bc $(g \cdot \rho)(u,v) = g \rho(g^{-1} u , g^{-1}v) $.  \ec Clearly, for $R, S \in \+ N(\hat F_\omega)$ of infinite index in $\hat F_\omega$, we have  \[ \hat F_\omega/ R \cong \hat F_\omega/S \lra \ex g \in V  \, [g \cdot \rho_R = \rho_S]. \] 

A homeomorphism of Cantor space is given by its action on the dense countable Boolean algebra $D$ of clopen sets. So $V$ is continuously isomorphic to a  closed subgroup of $\mathtt{Aut}(D)$, and hence of $S_\infty$, as required.  
\end{proof}

\begin{lemma}\label{lem:subgroup} Let $V$ be a closed subgroup of $S_\infty$  with a   Borel action  on a Polish space $X$. Then $E^X_V$ is classifiable by countable structures. \end{lemma} 

\begin{proof}  The proof is in two steps. The first step  adapts the  proof of   Becker and Kechris \cite[2.7.4]{Becker.Kechris:96} in the version of Gao \cite[3.6.1]{Gao:09}.

For a Polish group $G$ and Borel actions of $G$ on Polish space $X, Y$, a \emph{Borel $G$-embedding} is a Borel map $\theta \colon X \to Y $ such  that $g \cdot \theta(x) = \theta (g \cdot x)$ for each $x \in X$.  
 Let $\+ F(G)$ be  the usual Effros space consisting   of the closed subsets of  $G$.  By  \cite[2.6.1]{Becker.Kechris:96} $\+ F(G)^\omega$, with the canonical Borel $G$-action,    is universal for Borel  $G$-actions under Borel $G$-embeddings. 
 
For a countable signature $\+ L$, let $\mathtt{Mod}(\+ L)$ denote the Polish space of $\+ L$-structures with domain $\omega$. Note that $S_\infty$ acts continuously on $\mathtt{Mod}(\+ L)$. In  Gao \cite[3.6.1]{Gao:09}  an  (infinitary) signature  $\+ L_1$ is provided, together with a Borel $S_\infty$-embedding of $\+ F(S_\infty)^\omega$ into $\mathtt{Mod}(\+ L_1)$. Restricting this yields a Borel $V$-embedding $\theta$ of $\+ F(V)^\omega$ into $\mathtt{Mod}(\+ L_1)$.

 For the second step of the proof we use  the construction in \cite[Section 1.5]{Becker.Kechris:96} (also see the discussion around \cite[2.7.4]{Becker.Kechris:96}). For a  signature $\+ L_2$, which we   may  assume to be disjoint from $\+ L_1$, this yields an $\+ L_2$-structure $M$  with domain $\omega$ such that $\mathtt{Aut}(M) = V$.   
Given $x \in  X = \+ F(V)^\omega$,   let $N_x$ be the $\+ L_1 \cup \+ L_2$-structure with domain $\omega$ which has all the relations and functions of $\theta(x) $ and $M$. If   $\rho \in S_\infty$  shows that  $N_x \cong N_y$ then $\rho$ is  an automorphism of $M$, and hence  $\rho \in V$. Therefore $x E^X_V y \lra N_x \cong N_y$. 
 
 The lemma follows because the Borel action of $V $ on $\+ F(V)^\omega$ is universal for $V$-embeddings.
\end{proof}

\section{Isomorphism between profinite groups is Borel complete for $S_\infty$ orbit equivalence relations}

Recall that a  group $G$ is nilpotent of class 2 (nil-2 for short) if it satisfies the law $[[x,y],z]=1$. Equivalently, the commutator subgroup  is contained  in the center. For a prime $p$, the group of unitriangular matrices $$\mathrm{UT}^3_3(\ZZ / p \ZZ) = \big \{ \left(\begin{matrix} 1 & a  & c \\  0  &  1 & b \\ 0 &  0 & 1
\end{matrix}\right) \colon \, a,b,c \in \ZZ / p \ZZ\big \}$$ is an example of a nil-2 group of exponent~$p$.

\begin{theorem} \label{thm:Mekler} Let $p \ge 3$ be a prime. Any $S_\infty$ orbit equivalence relation can be Borel reduced to isomorphism between profinite nil-2   groups  of exponent~$p$.  
\end{theorem}
\begin{proof} The main result in   Mekler \cite{Mekler:81} implies  a version of Thm.\ \ref{thm:Mekler} for  abstract,   rather than profinite, groups.   Mekler associates to each symmetric and irreflexive graph $A$ a nil-2 exponent-$p$   group $G(A)$ in such a way that isomorphic graphs yield isomorphic groups.  In the countable case, the map $G$ sends    a   countable graph  $A$ to   a  countable  group  $G(A)$ in a Borel way.  
\begin{definition} A symmetric and irreflexive graph is  called \emph{nice} if it has no triangles, no squares, and for each pair of distinct vertices $x,y$, there is a vertex $z$ joined to $x$ and not to $y$.   \end{definition} 
 Mekler \cite{Mekler:81} proves that   a nice graph   $A$ can be interpreted in $G(A)$ using first-order formulas  without parameters (see \cite[Ch.\ 5]{Hodges:93} for background on interpretations). In particular,  for nice graphs $A, B$  we have  $A \cong B$ iff $G(A) \cong G(B)$. 
(For a more detailed  write-up see \cite[A.3]{Hodges:93}.) Since isomorphism of nice graphs is $S_\infty$-complete, so is isomorphism of countable nil-2 exponent $p$ groups.

Our proof is based on Mekler's, replacing the groups $G(A)$ he defined by their completions $\hat G(A)$  with respect to a suitable basis of nbhds of the identity. 

In the following all graphs will be  symmetric and irreflexive, and have  domain $\omega$. Such a graph is thus given by its set of edges   $A \sub \{ \la r, s \ra \colon \, r < s \}$.  We write $r A s$ (or   simply $rs$ if $A$ is understood)  for $\la r, s \ra \in A$. 

Let $F$ be the free nil-2 exponent-$p$ group on free generators $x_0, x_1, \ldots$. For $r \neq s$ we write $$ x_{r,s} = [x_r, x_s].$$ As noted in \cite{Mekler:81}, the centre $Z(F)$ of $F$ is an abelian group of exponent $p$ that is freely generated by the $x_{r,s}$ for $r< s$. Given a graph $A$, Mekler   lets 
\[ G(A) = F/ \la x_{r,s} \colon \, r A s \ra. \]
In particular $F = G(\ES)$. The centre  $Z= Z(G(A))$ is an abelian group of exponent $p$ freely generated by the $x_{r,s}$ such that  $\lnot r s$.  Also $G(A)/Z $ is 
an abelian group of exponent $p$ freely generated by the $Z x_i$. (Intuitively,  when  defining $G(A)$ as a quotient of $F$,   the commutators $x_{r,s}$ such that  $r A s$ vanish, but no   vertices vanish.)  

\begin{lemma}[Normal form for $G(A)$, \cite{Mekler:81,Hodges:93}] \label{lem: NF G}Every element $c$ of $Z$ can be written uniquely in the form $ \prod_{\la r, s \ra \in L} x_{r,s}^{\beta_{r,s}}$ where    $L \sub \omega \times \omega$ is a finite set of pairs $\la r, s \ra$ with $r<s$ and  $\lnot rAs$, and $0 < \beta_{r,s}  < p$. 

Every element of $G(A)$ can be written uniquely in the form $c \cdot v$  where  $c \in Z$, and $v=  \prod_{i \in D} x_i^{\alpha_i}$, for $D \sub \omega$   finite and $0 < \alpha_i < p$.  (The product  $\prod_{i \in D} x_i^{\alpha_i}$ is interpreted    along   the indices  in ascending order.)
 
\end{lemma}

Given  a graph $A$, let $R_n$ be the normal subgroup of $G(A)$   generated by the $x_i$, $i \ge  n$.    Note that $G(A)/R_n$ is finite, being a f.g.\  nilpotent torsion group.   Let    $  \widehat G(A)$ be  the   completion   of  $G(A)$    with respect to  the   set  $\+ V = \{R_n \colon \, n \in \omega\}$     (see Subsection~\ref{ss:completion}). By Lemma~\ref{lem: NF G} we have $\bigcap_n R_n= \{1\}$, so $G(A)$ embeds   into $\hat G(A)$. 
 
In  set theory one inductively defines $0 = \ES$ and $n = \{0, \ldots, n-1\}$ to obtain the  natural numbers; this will  save on notation here. A set of coset representatives for $G(A)/R_n$ is given by the $c\cdot v$  as in  Lemma~\ref{lem: NF G}, where  $D \sub n$ and $E \sub n \times n$.  The completion $\hat G(A)$  of $G(A)$ with respect to the $R_n$ consists of the maps $\rho  \in \prod_n G(A)/R_n$ such that $\rho (g R_{n+1}) = gR_n$ for each $n \in \omega$ and $g \in G(A)$. 


If $\rho (g R_{n+1}) = hR_n$ where $h = c \cdot v$ is a coset representative for $R_n$, then we can define a coset representative  $c' \cdot v'$ for $gR_{n+1}$   as follows: we obtain $c'$   from $c$ by potentially appending to  $c$ factors involving  the $x_{r,n}$ for $r< n$, and  $v'$ from $v$ by potentially appending a factor    $x_n^{\alpha_n}$.   So we can view $\rho $ as given by multiplying two  formal infinite products:

\begin{lemma}[Normal form  for $\hat G(A)$] 
Every   $c \in Z(\hat G(A))$ can be written uniquely in the form $ \prod_{\la r, s \ra \in L} x_{r,s}^{\beta_{r,s}}$ where    $L \sub \omega \times \omega$ is a   set of pairs $\la r, s \ra$ with $r<s$, $\lnot rAs$, and $0 < \beta_{r,s}  < p$.

 Every element of $\hat G(A)$ can be written uniquely in the form $c \cdot v$, where $v =  \prod_{i \in D} x_i^{\alpha_i}$,    $c \in Z(\hat G(A))$, $D \sub \omega$, and $0 < \alpha_i < p$ (the product is taken along  ascending indices).  

 \end{lemma}

We can define   the infinite products above explicitly as limits in $\hat G(A)$. We view $G(A)$ as embedded into $\hat G(A)$. Given formal products as above, let \bc $v_n =\prod_{i \in D\cap n } x_i^{\alpha_i}$   and  $c_n = \prod_{\la r, s \ra \in L \cap n \times n} x_{r,s}^{\beta_{rs}}$. \ec
For $k \ge n$ we have  $v_k v_n^{-1} \in \ol {R_n}$ and $c_k c_n^{-1} \in \ol {R_n}$. So   $v = \lim_n  v_n $ and $c = \lim_n c_n $ exist and equal the values of the formal products as defined above. 
 
Each nil-2 group satisfies the distributive law $[x,yz] = [x,y][x,z]$.  This implies that
$[x_r^\alpha, x_s^\beta] = x_{r,s}^{\alpha \beta}$. The following lemma generalises to infinite products the expression for commutators  that were obtained using these identities in   \cite[p.\ 784]{Mekler:81} (and also in \cite[proof of Lemma A.3.4]{Hodges:93}).

\begin{lemma}[Commutators] \label{lem:com} The following holds in $\hat G(A)$.  \[[\pp \alpha r D,  \pp \beta s E] = \prod_{r \in D, \, s \in E, \,  r< s, \, \lnot rs} x_{r,s}^{\alpha_r \beta_s- \alpha_s \beta_r}\]\end{lemma}  

\begin{proof} Based on the  case of finite products, by continuity of the commutator operation and  using the expressions for limits above, we have 
\begin{eqnarray*} [\pp \alpha r D,  \pp \beta s E] &=& [\lim_n \pp \alpha r {D \cap n},  \lim_n \pp \beta s {E\cap n}] \\ 
&=&  \lim_n \prod_{r \in D, \, s \in E, \,  r< s< n, \, \lnot rs} x_{r,s}^{\alpha_r \beta_s- \alpha_s \beta_r} \\ 
&=& \prod_{r \in D, \, s \in E, \,  r< s, \, \lnot rs} x_{r,s}^{\alpha_r \beta_s- \alpha_s \beta_r}.\end{eqnarray*}
\end{proof}
	The following is a direct consequence of Lemma~\ref{lem:com}.   $C(g)$ denotes the  centraliser of a group element $g$.
 \begin{lemma} \label{lem:centr} Let $  v \in \hat G(A) $. If  $0<  \gamma <p$ we have $C(  v^\gamma)= C(  v)$. \end{lemma}
Mekler's  argument employs the niceness of $A$ to show that a  copy of the set of vertices  of the given graph is first-order definable in $G(A)$. The copy of  vertex~$i$ is a certain definable equivalence class of the generator~$x_i$. 
 He provides  a first-order  interpretation $\Gamma_2 $ without parameters such that $\Gamma_2(G(A)) \cong A$. We will show  that his interpretation has the same effect  in the profinite case: $\Gamma_2(\hat G(A)) \cong A$.

 We first summarize Mekler's interpretation $\Gamma_2$Á. Let $H$ be a group with centre  $Z(H)$. \bi \item For $a \in H$ let $\bar a$ denote the  coset $a Z(H)$.  
 \item Write $\bar a \sim \bar b$ if $C(a)= C(b)$. Let $[\bar a]$ be the $\sim$ equivalence class of $\bar a$.  (Thus, for $c \in Z(H)$ we have $[\bar c] = \{ Z(H)\}$.)  \item 
 
Slightly deviating from Mekler's notation, let   \bc $[\bar a ] R [\bar b] \lra  [\bar a], [\bar b], [\bar 1]$ are all distinct and  $[a,b]=1$.  \ec\ei Let  
 
 $$\Gamma_1 ( H) = \la ( \{ [\bar a] \colon a   \in H \setminus  Z(H)\}, R \ra.$$
 
 $$\Gamma_2(H) = \{ [ \bar a] \colon \,  |[\bar a]| = p-1 \lland \ex \bar w  \,  [ \bar  a ] R  [\bar w]  \}, $$ viewed as a subgraph of $\Gamma_1 ( H)$. 
 
We are ready to verify that a nice graph can be recovered from its profinite group via the interpretation $\Gamma_2$. 
  \begin{lemma} For a nice graph $A$, we have $\Gamma_2(\hat G(A)) \cong A$ via $[\ol {x_i}] \mapsto i$. \end{lemma}

We  follow the outline of  the  proof of \cite[Lemma 2.2]{Mekler:81}. The notation there will be adapted  to $\hat G(A)$ via allowing infinite products. 

Henceforth,   in products of the  form $\pp \alpha i D$   we will assume that $D\sub \omega$ is nonempty and $0 < \alpha_i < p$ for each $i \in D$. We also set $\alpha_i=0$ for $i \not \in D$. Expressions and  calculations involving number exponents $\alpha$  will all be modulo $p$; e.g.\ $\alpha \neq 0$ means that $\alpha \not \equiv 0 \mod p$.

 All we need to show is that the formula in   $\Gamma_2$ defines the set $\{[\bar x_i] \colon i \in \omega \}$. Suppose some $\bar v$ is given where $v \not \in Z(\hat G(A))$. We may assume that    $  v= \pp \alpha  i D$ satisfying  the conventions above.   We will show that   \bc $ [\bar v] = [\bar x_k]$ for some $k \lra |[\bar v]| = p-1 \lland \ex \bar w  \,  [ \bar  v ] R  [\bar w]$.  \ec
 We distinguish four cases, and verify the equivalence in each  case.  In Case 1 we check that the formula  holds, in Case 2-4 that it  fails. 
 Recall that $rs$ is short for $rA s$, i.e.\ that vertices $r,s$ are joined. 

 \subsection*{Case 1:   {\rm $D = \{ r \}$ for some $r$}} \
 
For the first condition of the formula in $\Gamma_2$, note that we have $\bar v= \bar x_r^\alpha$ for some $\alpha \neq 0$. Suppose $\bar u \sim \bar v$ for $  u = \pp \beta k E $. Then $E = \{r\}$. For,  if $k \in E$, $k \neq r$, then by niceness of $A$ there is an  $s$ such that $rs  \lland \lnot   ks$,  so that $  x_s \in C(v) \setminus C(w)$ by  Lemma \ref{lem:com}. Thus $\bar u = \bar x_r^\beta$ with    $0< \beta < p$.

For the second condition, by niceness of $A $ we   pick $k \neq i$ such that $ik$, and let $\bar w = \bar x_k$.
   \subsection*{Case 2: {\rm   $D = \{r, s\}$ for some $r,s$  such that   $rAs$ (and hence $r \neq s$)}}  
   
   \
   
 \n  We show that   $[\bar{ v}] = \{ \bar x_r^\alpha \bar x_s^\beta \colon 0< \alpha, \beta < p\}$, and hence this equivalence class has size $(p-1)^2$. To this end we verify:

    \n \emph{Claim.}    Let $w= \pp  \beta k E$. For each $\alpha, \beta \neq 0$, we have \bc  $[w, \bar x_r^\alpha \bar  x_s^{\beta} ]= 1 \lra E \sub \{r,s\}$.  \ec
  For the implication from left to right,   if there is $k \in E\setminus \{r,s\}$, then $kr $ and $ks$, so   $A$ has a triangle.  Hence $E \sub \{r,s\}$.
  The converse implication follows by distributivity  and since $[x_r,x_s]=1$. This shows the claim.

   
   \subsection*{Case 3:  {\rm  Neither Case 1 nor 2, and there is an $\ell$ such that $i\ell$ for each $i \in D$.} }  \
 
  \n \emph{Claim.}     $[\bar{ v}] = \{ \bar v^\gamma \bar x_\ell^\beta \colon 0 < \gamma < p \lland 0\le   \beta < p\}$, and hence has size $p(p-1)$.

\n First an observation: \emph{suppose that $[v^\gamma  x_\ell^\beta,w]=1$ where  $\gamma, \beta$ are as above and $w= \pp \beta k E$. Then  $E \sub D \cup \si l$.} For, since   Case 1 and  2 don't apply, there are distinct $p,q \in D \setminus \si l$. Since $A$ has no squares,  $\ell$ is the only vertex adjacent to both $p$ and $q$. Hence, given  $j \in E \setminus \{ \ell,p,q\}$, we have $\lnot pj \lor \lnot qj$, say the former. Then $j \in D$, for otherwise, in Lemma~\ref{lem:com}, in the expansion of $[v^\gamma x_\ell^\beta, w]$, we get a term $x_{p,j}^{m}$ with  $m \neq 0$ and $\lnot pj$ (assuming that $p< j$, say).

  The inclusion  ``$\supseteq$" of the claim now follows  by Lemma~\ref{lem:com}.
    %
    For the inclusion ``$\sub$"    suppose that $[v,w]=1$ where $w= \pp \beta k E$. Then  $E \sub D \cup \si l$ by our observation.       
  By Lemma~\ref{lem:com} we now have 
  \bc $[v,w] = \prod_{r,s \in D \setminus \si \ell, \,  r< s, \, \lnot rs} x_{r,s}^{\alpha_r \beta_s- \alpha_s \beta_r}$. \ec
Let  $m = \min (D \setminus \si \ell)$. Since $\alpha_m \neq 0$ we can pick $\gamma$ such that $\beta_m= \gamma \alpha_m$. Since $A$ has no  triangles we have $\lnot rs$ for any $r< s$ such that  $r, s \in D \setminus \si \ell$, and hence $\alpha_r \beta_s = \alpha_s \beta_r$.  By induction on the elements $s$ of $D \setminus \si \ell$ this yields $\beta_s= \gamma \alpha_s $: if we have it for some  $r<s$ in $D \setminus \si \ell$ then $\beta_r \beta_s = \gamma \alpha_r\beta_s = \gamma \alpha_s \beta_r$.  Hence 
  $\beta_s= \gamma \alpha_s $, for if $\beta_r\neq 0$ we can cancel it, and if $\beta_r = 0$ then also $\beta_s=0$ because $\alpha_r \neq 0$. 
  
  This shows that $\bar w=  \bar v^\gamma \bar x_\ell^\beta$ for some $\beta$. In particular, $C(v)= C(w)$ implies that $\bar w$ has the required form.

     \subsection*{Case 4: {\rm   Neither Case 1, 2 or 3}} \
  
  \n \emph{Claim.}  $[\bar{ v}] = \{ \bar v^\gamma   \colon 0 < \gamma < p\}$, so this class  has $p-1$ elements.  \\ There is no $\bar w$ such that $[\bar{ v}] R [\bar{ w}]$.
  
 The inclusion ``$\supseteq$" of the first statement follows from Lemma~\ref{lem:centr}. We now verify the converse inclusion and the second statement. By case hypothesis  there are distinct $\ell_0, \ell_1\in D$ such that $\lnot \ell_0 \ell_1$. Since $A$ has no squares  there is at most one $q \in D$ such that $\ell_0q \lland \ell_1 q$.  If $q$ exists, as we are not in Case 3 we can choose $q' \in D$ such that $\lnot q' q$.  
 
 We define a linear order $\prec$ on  $D$, which is  of type $\omega$ if $D$ is infinite. It begins  with  $\ell_0, \ell_1$, and is followed by $q', q$ if they are defined. After that we proceed in ascending order for the remaining elements of $D$. Then  for each $v \in D \setminus \si {\min D}$ there is $u \prec v$ in $D$ (in fact among the first three elements)  such that $\lnot uv$.
  
  Suppose now that $[v,w]=1$ where $w= \pp \beta k E$. Then $E \sub D$: if $s \in E \setminus D$ there is $r \in D$ such that  $\lnot rs$. This implies $\alpha_r \beta_s= \alpha_s \beta_r$, but $\alpha_s= 0$ while the left hand side is $\neq 0$, contradiction.
  
  By a slight variant of Lemma~\ref{lem:com}, using that $\prec$ eventually agrees with $<$,   we now have $[v,w] = \prod_{r,s \in D ,  \,  r\prec  s, \, \lnot rs} x_{r,s}^{\alpha_r \beta_s- \alpha_s \beta_r}$. Choose $\gamma$ such that $\gamma \alpha_{\ell_0} = \beta_{\ell_0}$. By induction along $(D, \prec)$ we see that $\beta_s = \gamma \alpha_s$ for each $s\in D$: if $\ell_0 \prec s$ choose $r \prec s$ such that $\lnot rs$.  Since $\alpha_r \beta_s= \alpha_s \beta_r$,  as in Case 3 before we may conclude that $\beta_s = \gamma \alpha_s$.
  
  This shows that $\bar v^\gamma = \bar w$. 
   Further, if   $[\bar w] \neq [\bar 1]$   then $\gamma \neq 0$ so that  $[\bar v] = [\bar w]$, as required. 
\end{proof}

 \def\cprime{$'$}

%
%
%

\end{document}